\documentclass[12pt,a4paper]{amsart}
\usepackage{url}
\usepackage{amsmath,amsthm,amsfonts,amssymb,latexsym}

\newtheorem{theorem}{Theorem}[section]

\newtheorem{lemma}[theorem]{Lemma}

\newtheorem{corollary}[theorem]{Corollary}

\theoremstyle{definition}

\newtheorem{example}[theorem]{Example}
\newtheorem{question}[theorem]{Question}

\newcommand{\spl}{\mathrm{Split}}

\newcommand{\U}{\mathcal U}
\newcommand{\w}{\omega}

\newcommand{\IQ}{\mathbb Q}

\newcommand{\IP}{\mathbb P}

\newcommand{\A}{\mathcal{A}}

\newcommand{\F}{\mathcal{F}}

\newcommand{\V}{\mathcal{V}}

\newcommand{\vid}{\hat{\ \ }}
\newcommand{\bigvid}{\hat{\ \ }}

\newcommand{\uhr}{\upharpoonright}

\newcommand{\name}[1]{\dot{#1}}
\newcommand{\la}{\langle}
\newcommand{\ra}{\rangle}

\newcommand{\bM}{\mathcal M}

\newcommand{\forces}{\Vdash}

\newcommand{\hot}{\mathfrak}

\newcommand{\nothing}[1]{}

\newcommand{\zrost}{\w^{\uparrow\w}}

\title[Well-splitting posets]{On well-splitting posets}

\author[D. Repov\v{s}  and L. Zdomskyy]{Du\v{s}an Repov\v{s}  and Lyubomyr Zdomskyy}

\address{Faculty of Education, and Faculty of Mathematics and Physics,
University of Ljubljana, and Institute of Mathematics, Physics and Mechanics, Ljubljana, Slovenia 1000.}
\email{dusan.repovs@guest.arnes.si}
\urladdr{http://www.fmf.uni-lj.si/\~{}repovs/index.htm}

\address{Universit\"at Wien, Institut f\"ur Mathematik, Kurt G\"odel Research Center, Augasse 2--6, UZA 1 - Building 2, 1090 Wien, Austria.}
\email{lzdomsky@gmail.com}
\urladdr{http://www.logic.univie.ac.at/\~{}lzdomsky/}


\thanks{The first author was partially supported by
the Slovenian Research Agency grants P1-0292, N1-0114, and N1-0083.
The second author would
like to thank  the Austrian Science Fund FWF (Grants I2374-N35 and I3709-N35)
 for generous support for this research. \\
 }

\begin{document}
\begin{abstract}
We introduce a class of proper posets which
 is preserved under countable support iterations,  includes $\w^\w$-bounding, Cohen, Miller, and
Mathias posets associated to filters with the Hurewicz covering properties, and has the property
 that the ground model reals remain splitting and unbounded
in corresponding extensions.
Our results  may be considered as a possible
path towards solving  variations of the famous Roitman problem.
\end{abstract}

\maketitle

\section{Introduction}

The famous  Roitman problem asks whether it is consistent that $\hot d=\w_1<\hot a$.
Here, $\hot d$ is the minimal cardinality of a subfamily of $\w^\w$ which is \emph{dominating} with respect to
the preorder relation $\leq^*$ on $\w^\w$, where
$a\leq^*b$ for $a,b\in\w^\w$
 means that $a(n) \leq b(n)$
for all but finitely many $n$; and
$\hot a$ is the minimal cardinality of an infinite \emph{mad subfamily $\mathcal A$} of $[\w]^\w$, i.e.,
an infinite subfamily whose distinct elements have finite intersection and which is maximal with respect to this property.

Without  the restriction $\hot d=\w_1$, the consistency of $\hot d<\hot a$ has been established in a breakthrough work
of Shelah \cite{She_cup}. Regarding the original Roitmann problem, even the following weaker version thereof raised in \cite{Bre17}  remains open:
Is it consistent that $\hot s=\hot b=\w_1<\hot a$?  Here, $\hot s$ is the minimal cardinality of a \emph{splitting} family, i.e., a family $\mathcal S\subset [\w]^\w$ such that
for every $X\in [\w]^\w$ there exists $S\in\mathcal S$ for which both $S\cap X$ and $X\setminus S$ are infinite; and
$\hot b$ is the minimal cardinality of a subfamily of $\w^\w$ which is \emph{unbounded} with respect to $\leq^*$. It is well-known that
$\max\{\hot b,\hot s\}\leq\hot d$ and the strict inequality holds, e.g., in the Cohen model (see
 \cite{Bla10, Vau90} for more information on these and many other cardinal characteristics of the continuum).

In this paper we isolate the class of \emph{well-splitting}  posets (see the next section for the definition) with properties described in the abstract, aiming at the solution of the aforementioned weak version of Roitman's problem. This class includes among others Mathias posets associated to filters on $\w$ with the Hurewicz covering property. This motivates the following
\begin{question} \label{Qu}
(CH) Can every mad family be destroyed by a well-splitting poset? In particular, given a mad family $\mathcal A$, is there a well-splitting poset
$\IP$ such that in $V^{\IP}$, $\{\w\setminus A:A\in\mathcal A\}$ can be enlarged to a Hurewicz filter, or more generally to a filter, whose Mathias forcing is well-splitting?
\end{question}
By our Theorem~\ref{main} proved in the next section, the affirmative answer to Question~\ref{Qu} would allow one to construct a model of
$\hot b=\hot s=\w_1<\hot a =\w_2$.

Recall from \cite{Hur27} that a topological space
$X$ is said to have the \emph{Hurewicz covering  property} (or is simply called Hurewicz space) if for every sequence $\la \U_n : n\in\omega\ra$
of open covers of $X$ there exists a sequence $\la \V_n : n\in\omega \ra$ such that
each $\V_n$ is a finite subfamily of $\U_n$ and the collection $\{\cup \V_n:n\in\omega\}$
is a $\gamma$-cover of $X$, i.e.,  the set $\{n\in\w:x\not\in\cup\V_n\}$ is finite for each $x\in X$. It is clear that
$\sigma$-compact spaces are Hurewicz, but by
 \cite[Theorem~5.1]{COC2} there  also exist non-$\sigma$-compact sets of reals  having the Hurewicz property. We consider  each filter on $\w$ with the subspace topology inherited from $\mathcal P(\w)$, the latter being a topological copy of the Cantor space $2^\w$ via characteristic functions.
  As it was proved in \cite{ChoRepZdo15}, the Mathias forcing associated to a filter $\F $ is almost
 $\w^\w$-bounding in terminology of  \cite{She84} if and only if $\F$ is Hurewicz. It is worth mentioning here that in general, almost $\w^\w$-bounding posets can
 make ground model reals non-splitting, see, e.g., \cite[Lemma~1.14]{She84}, so by our
 Lemma~\ref{obs_well_spl} almost $\w^\w$-bounding posets do not have to be well-splitting.

 Building on the proof of \cite[Theorem~3.1]{Bre98},  it was established in \cite{Zdo??} that under CH, for every
 mad family $\A$, the collection  $\{\w\setminus A:A\in\mathcal A\}$ can be enlarged to an ultrafilter  $\F$ with a certain covering property
 which is
 weaker (but similar) to the Hurewicz one, and whose Mathias forcing does not produce any new real dominating the given ground model unbounded subset.
The construction in \cite{Zdo??} cannot be directly used to answer Question~\ref{Qu} since
by
Lemma~\ref{obs_well_spl}
the Mathias forcing for ultrafilters cannot be well-splitting because it adds an unsplit real. However, it is natural to ask how far can one weaken the
Hurewicz property of a filter  so that its Mathias forcing is still well-splitting.

\begin{question}
Let $\F$ be a filter on $\w$ whose Mathias forcing is well-splitting. Is then $\F$ Hurewicz? In other words,
are  well-splitting and almost $\w^\w$-bounding equivalent for such posets?
\end{question}

\section{Well-splitting posets}

Throughout this section we denote by $E_0$ and $E_1$ the sets
of all  even  and odd natural numbers, respectively.
A strictly increasing function $f\in\w^\w$ is said to \emph{well-split} a set $M$ if
the sets $\{n\in E_j: |[f(n), f(n+1))\cap Y|\geq 2\}$ are infinite
 for all $Y\in M\cap [\w]^\w$ and $j\in 2$. A simple diagonalization argument
 shows that
 for every countable $M$ there is a function
 well-splitting it.

We shall say that a poset $\IP$ is \emph{well-splitting} if the following is satisfied: Whenever
$\IP \in M$, where $M$ is a countable elementary submodel of $H(\theta)$ for any
sufficiently large $\theta$, $p\in M\cap P$ and $f$ well-splits $M$,
 then
there is some $q\leq p$
which is $(M, \IP)$-generic and such that $q$ forces
$f$ to well-split $M[\Gamma]$, where $\Gamma$ is the canonical name for $\IP$-generic filter.

\begin{lemma} \label{obs_well_spl}
Supose that $\IP$ is well-splitting and $G$ is $\IP$-generic.
 Then $V\cap [\w]^\w$ is splitting and $V\cap\w^\w$ is unbounded in $V[G]$.
\end{lemma}
\begin{proof}
To see that $\w^\w\cap V$ is unbounded, let us fix a $\IP$-name
$\name{h}$ for an element of $\zrost$ (the family of all strictly
increasing functions in $\w^\w$),   $p\in\IP$, and
pick a countable elementary submodel
$M$ of $H(\theta)$  such that $\IP,\name{h},p\in M$.
Suppose that  $f$ well-splits $M$ and  $q\leq p$ is any  $(M,
\IP)$-generic condition which forces    $f$ to well-split
$M[\Gamma]$. Let $\name{h}_1\in M$ be a $\IP$-name for the following
function: $\name{h}_1(0)=1$,
$\name{h}_1(n+1)=\name{h}(\name{h}_1(n))+1$ for all $n\in\w$.
It
follows from the above that  $q$ forces the set
$$ \name{I}:=\{n\in E_0: |[f(n),f(n+1))\cap\mathrm{range}(\name{h}_1)|\geq 2\} $$
to be infinite. Let $G\ni q$ be $\IP$-generic and set
$I=\name{I}^G$, $h=\name{h}^G$, and $h_1=\name{h}_1^G$.
Let us note that
$[f(\min I),f(\min I+1))\cap\mathrm{range}(h_1)|\geq 2$ yields
$$f(\min I+1)\geq f(\min I)+2\geq\min I+2=(\min I+1)+1,$$
and therefore  $f(i)\geq i+1$ for every $i>\min I$, because $f$ is strictly increasing.

In $V[G]$,
for every $i\in I\setminus\{\min I\}$ we can find $n_i\in \w $ such that
$h_1(n_i),h_1(n_i+1)\in [f(i), f(i+1))$. Thus\footnote{The second inequality follows from
$h_1(n_i)\geq f(i)\geq i+1$.}
$$
h_1(n_i+1)=h(h_1(n_i))+1<f(i+1)\leq  f(h_1(n_i)),$$
 i.e.,
$$ \big\{h_1(n_i):i\in I\setminus\{\min I\}\big\}\subset\{k:h(k)<f(k)\}, $$
and hence $h$ does not dominate $f$. Summarizing the above, we
conclude that for any $p\in\IP$ and  $\IP$-name $\name{h}$ for an
element of $\zrost$, there is a stronger condition $q$ and
$f\in\zrost\cap V$ such that $q$ forces the set
$\{k:\name{h}(k)<f(k)\}$ to be infinite. This  means precisely  that
$\zrost\cap V$ is unbounded in $V[G]$ for any $\IP$-generic filter
$G$.

To prove  that $[\w]^\w\cap V$ is splitting, let us fix a $\IP$-name
$\name{Y}$ for an element of $[\w]^\w$,  $p\in\IP$,
and
pick a countable elementary submodel
$M$ of $H(\theta)$  such that $\IP,\name{Y},p\in M$.
 Suppose that  $f$ well-splits $M$ and  $q\leq p$ is
any  $(M, \IP)$-generic condition which forces    $f$ to
well-split $M[\Gamma]$. Then  $q$ forces the sets
$$ \name{I}_j:=\bigcup_{n\in E_j}[f(n),f(n+1))\cap\name{Y} $$
to be infinite for all $j\in 2$. Since the sets $\bigcup_{n\in
E_j}[f(n),f(n+1))$, $j\in 2$, are disjoint, infinite, and both are
in $V$, this completes our proof.
\end{proof}
Note that the converse to Lemma~\ref{obs_well_spl} does not hold.
\medskip

\begin{example} \label{referee}
There exists a ccc non-well-splitting poset which
preserves ground model reals as a splitting and unbounded family. Indeed, let $\mathbb C_{\w_1}$
be the poset obtained by adding $\w_1$-many Cohen reals with countable supports and $G$ a $\mathbb C_{\w_1}$-generic.
Let $\U\in V$ be an ultrafilter on $\w$  and $\U'=\{X\in\mathcal P(\w)^{V[G]} : \exists U\in \U(U\subset X) \}$
be the filter in $V[G]$ generated by $\U$ as its base. We claim that in $V[G]$ the poset $\IQ:=\mathcal M_{\U'}$
preserves ground model reals (i.e., $([\w]^\w)^{V[G]}$) as a splitting and unbounded family
in $V[G*H]$, where $H$ is $\IQ$-generic over $V[G]$.
Indeed,  by a rather standard argument similar to that in the proof of \cite[Theorem~11]{SchTal10},
one can check that $(\U')^{<\w}$ is $+$-Ramsey in $V[G]$ in the sense of \cite{Laf96}, i.e.,
for every sequence $\la X_n:n\in\w\ra$ in $((\U')^{<\w})^+$ there is a selector $\la a_n\in X_n:n\in\w\ra$
such that $\{a_n:n\in\w\}\in ((\U')^{<\w})^+$. By \cite[Prop.~1 and Th.~19]{ChoGuzHru16} we get that
$(\w^\w)^{V[G]}$ is non-meager in $(\w^\w)^{V[G*H]}$, and thus  $(\w^\w)^{V[G]}$ is unbounded in $V[G*H]$, and
it is also easy to see that $(\w^\w)^{V[G]}\cup\big\{\{\{n\}\times\w\}:n\in\w\big\}\subset\mathcal P(\w\times\w)^{V[G]}$
is a splitting family of subsets of $\w\times\w$ in $V[G*H]$.

On the other hand, $\mathbb C_{\w_1}*\IQ$ adds a pseudointersection to $\U$ and hence $([\w]^\w)^V$ is not a splitting family in $V[G*H]$, which together with Lemma~\ref{cs-pres} and Corollary~\ref{cohen}  implies that
$\IQ$ is not well-splitting in $V[G]$.
\hfill $\Box$
\end{example}

It is clear that each well-splitting poset is proper and an iteration of finitely many well-splitting posets is again well-splitting.
Next, we shall establish that being well-splitting is also preserved by countable support iterations.
The proof of the following lemma is similar to that of \cite[Lemma~2.8]{Abr10}, with some additional control
on the sequence $\la\name{p}_i:i\in\w\ra$.

Let us make a couple of standard conventions regarding our notation.
Whenever  $\la \IP_\alpha,\name{\IQ}_\alpha:\alpha<\delta \ra $
is an iterated forcing construction, we denote by $\IP_{[\alpha_0,\alpha_1)}$ a
$\IP_{\alpha_0}$-name for the quotient poset $\IP_{\alpha_1}/\IP_{\alpha_0}$, viewed naturally
as an iteration over the ordinals $\xi\in\alpha_1\setminus\alpha_0$. For a $\IP_{\alpha_0}$-generic $G$
and a $\IP_{\alpha_1}$-name $\tau$, where $\alpha_0\leq \alpha_1$, we denote by
$\tau^G$ the $\IP_{[\alpha_0,\alpha_1)}^G$-name in $V[G]$ obtained from $\tau$ by partially interpreting it
with $G$. This allows us to speak about, e.g., $\IP_{[\alpha_1,\alpha_2)}^G$ for
$\alpha_0\leq\alpha_1\leq\alpha_2\leq\delta$  and a $\IP_{\alpha_0}$-generic filter
$G$. For a poset $\IP$ we shall denote the standard $\IP$-name
for  $\IP$-generic filter  by $\Gamma_{\IP}$. We shall write $\Gamma_\alpha$ instead of $\Gamma_{\IP_\alpha}$
whenever we work with an iterated forcing construction which will be clear from the context.
Also, $\Gamma_{[\alpha_0,\alpha_1)}$ is a $\IP_{\alpha_1}$-name whose interpretation with respect
to a $\IP_{\alpha_0}$-generic filter $G$ is  $\Gamma_{\IP_{[\alpha_0,\alpha_1)}^G}$, which is an element of $ V[G]$.

\begin{lemma} \label{cs-pres}
If $\la \IP_\alpha,\name{\IQ}_\alpha:\alpha<\delta \ra \in M$ is a countable support iteration of
well-splitting (hence proper) posets, then   $\IP$ is also well-splitting.
\end{lemma}
\begin{proof}
The proof is by induction on $\delta$. The successor case is clear.
So assume that $\delta$ is limit, $p\in\IP_\delta$,  and $M\ni \IP_\delta,p $
is a countable elementary submodel of $H(\theta)$ for a
sufficiently large $\theta$.  Pick an increasing sequence $\la\delta_i:i\in\w\ra$ cofinal in $\delta\cap M$,
with $\delta_i\in M$ for all $i\in\w$. Let also $\{D_i:i\in\w\}$ and $\{\name{Y}_i:i\in\w\}$ be an enumeration of all open dense subsets of $\IP_\delta$ and
all $\IP_\delta$-names for an infinite subset of
$\w$ which are elements of $M$, respectively. We can assume without loss of generality, that for
every $\IP_\delta$-name $\name{Y}\in M$ for an element of $[\w]^\w$ the set $\{i\in\w:\name{Y}=\name{Y}_i\}$ is infinite.
Suppose that $f$ well-splits $M$.
  By induction on $i\in\w $ we will define a condition
$q_i\in\IP_{\delta_i}$ and  $\IP_{\delta_i}$-names $\name{p}_i,
\name{n}^0_i,\name{n}^1_i$  such that
\begin{itemize}
\item[$(i)$] $\name{p}_i$ is a name for an element of $\IP_\delta$,  $q_0\forces_{\delta_0}\name{p}_0\leq \check{p}$,
and $q_{i+1}\forces_{\delta_{i+1}}\name{p}_{i+1}\leq\name{p}_i$;
\item[$(ii)$] $q_{i+1}\uhr\delta_i=q_i$;
\item[$(iii)$] $q_i$ is $(M,\IP_{\delta_i})$-generic;
\item[$(iv)$] $ \name{n}^0_i,\name{n}^1_i$ are $\IP_{\delta_i}$-names for natural numbers bigger than $i$; and
\item[$(v)$] $q_i$ forces over $\IP_{\delta_i}$ that ``$\name{p}_i\uhr\delta_i\in\Gamma_{\delta_i}$,
$\name{p}_i\in D_i\cap M$,
and $\name{p}_i$ forces over $\IP_\delta$  that $\name{n}^j_i\in E_j$ and $|[f(\name{n}^j_i), f(\name{n}^j_i+1))\cap\name{Y}_i|\geq 2 $
for all $j\in 2$''.
\end{itemize}
Suppose now that we have constructed objects as above and set $q=\bigcup_{i\in\w}q_i$.
Since $q_i=q\uhr\delta_i$ forces over $\IP_{\delta_i}$ that $\name{p}_i\uhr\delta_i\in\Gamma_{\delta_i}$
and $q_{i+1}\forces_{\delta_{i+1}}\name{p}_{i+1}\leq\name{p}_i$ for all $i$, a standard argument yields that
$q$ is $(M,\IP_\delta)$-generic and $q\forces_\delta\name{p}_i\in\Gamma_\delta$ for all $i\in\w$, see, e.g.,
the proof of \cite[Lemma~2.8]{Abr10} for details. Then $q$ forces that  $\tau_0:=\{n\in  E_0: |[f(n), f(n+1))\cap\name{Y}|\geq 2\}$ and $\tau_1:=\{n\in  E_1: |[f(n), f(n+1))\cap\name{Y}|\geq 2\}$ are infinite for any $\IP_\delta$-name
$\name{Y}$ for an infinite subset of $\w$: Given  $\IP_\delta$-generic $G\ni q$, note that $p_i:=\name{p}_i^G\in G$
for all $i$. Now $(v)$ implies $n^j_i\in\tau_j^G$ for all  $j\in 2$ and $i\in\w$ such that $\name{Y}=\name{Y}_i$, where $n^j_i=(\name{n}^j_i)^G$.

Returning now to the inductive construction, assume that
$q_i\in\IP_{\delta_i}$,  $\IP_{\delta_i}$-names $\name{p}_i, \name{n}^0_i,\name{n}^1_i$ satisfying $(i)$-$(v)$ have already been constructed.
Let $G_{\delta_i}$ be  $\IP_{\delta_i}$-generic containing $q_i$ and $p_i=\name{p}_i^{G_{\delta_i}}\in\IP_\delta\cap M$.
By  $(v)$ we know that $p_i\uhr\delta_i\in G_{\delta_i}$. In $V[G_{\delta_i}]$ let
 $p_i'\in M\cap D_{i+1}$ be such that $p_i'\leq p_i$ and $p_i'\uhr\delta_i\in G_{\delta_i}$.
By the maximality principle we get a $\IP_{\delta_i}$-name $\name{p}'_i$
for a condition in $\IP_\delta$ such that $q_i \forces_{\delta_i}$ ``$\name{p}'_i\leq \name{p}_i$,
 $\name{p}'_i\in M\cap D_{i+1}$, and $\name{p}'_i\uhr\delta_i\in\Gamma_{\delta_i}$''.

Given a  $\IP_{\delta_{i+1}}$-generic filter $R$,  construct in $V[R]$
a decreasing sequence  $\la r_m:m\in\w\ra\in M[R]$ of conditions in $\IP^R_{[\delta_{i+1},\delta)}$
below $(\name{p}'_i\uhr [\delta_{i+1},\delta))^R$
such that for some $a_m\in [\w]^m$ we have $r_m\forces_{\IP^R_{[\delta_{i+1},\delta)}} $ ``$a_m$ is the set of the first $m$ elements of $\name{Y}_{i+1}$''. By the maximality principle we get a sequence
$\la \rho_m:m\in\w\ra\in M$ of $\IP_{\delta_{i+1}}$-names for elements of
$\IP_{[\delta_{i+1},\delta)}$ such that
\begin{eqnarray*}\forces_{\delta_{i+1}} \big[\rho_{m+1}\leq\rho_m\wedge \exists \nu_m\in [\w]^m\: (\rho_m \forces_{\IP_{[\delta_{i+1},\delta)}} \\
  ``\nu_m\mbox{ is the set of the first $m$ many elements of  }\name{Y}_{i+1} \: '' )\big].
\end{eqnarray*}
In the notation used above, let
$\name{Z}$ be a $\IP_{\delta_{i+1}}$-name
for $\bigcup_{m\in\w}\nu_m$
and note that $\name{Z}$ is  a $\IP_{\delta_{i+1}}$-name for an infinite subset of $\w$.

Let again $G_{\delta_i}$ be  $\IP_{\delta_i}$-generic containing $q_i$ and $p'_i=(\name{p}'_i)^{G_{\delta_i}}\in\IP_\delta\cap M\cap D_{i+1}$.
It also follows from the above that $p'_i\uhr\delta_i\in G_{\delta_i}$. For a while we shall work in $V[G_{\delta_i}]$.
 Since
$\IP^{G_{\delta_i}}_{[\delta_i,\delta_{i+1})}$ is well-splitting in $V[G_{\delta_i}]$ by our inductive assumption,
there exists a $(M[G_{\delta_i}], \IP^{G_{\delta_i}}_{[\delta_i,\delta_{i+1})})$-generic condition $\pi\leq p_i'\uhr[\delta_i,\delta_{i+1})^{G_{\delta_i}}$
such that
$$ \pi\forces_{\IP^{G_{\delta_i}}_{[\delta_i,\delta_{i+1})}}
 \tau_j:=\{n\in  E_j: |[f(n), f(n+1))\cap\name{Z}^{G_{\delta_i}}|\geq 2\}$$
 is infinite for all $j\in 2. $
Let $H$ be $\IP^{G_{\delta_i}}_{[\delta_i,\delta_{i+1})}$-generic over $V[G_{\delta_i}]$ containing $\pi$
and $n^j_{i+1}\in\tau^H_{j}\setminus (i+2)$, where $j\in 2$. In $V[G_{\delta_i}*H]$ pick
$m\in\w$ such that
$$r_m:=\rho_m^{G_{\delta_i}*H} \forces_{\IP_{[\delta_{i+1},\delta)}^{G_{\delta_i}*H}} \name{Z}^{G_{\delta_i}*H}\cap f(\max_{j\in 2}(n^j_{i+1})+1)=
\name{Y}_{i+1}^{G_{\delta_i}*H}\cap f(\max_{j\in 2}(n^j_{i+1})+1).$$
In $M[G_{\delta_i}]$ pick a condition $s\in M[G]\cap H$ below $p_i'\uhr[\delta_i,\delta_{i+1})^{G_{\delta_i}}$
forcing the above properties of $n^j_{i+1}$, $\tau_j$, and $\rho_m$, where $j\in 2$.
By the maximality principle  we obtain $\IP_{[\delta_i,\delta_{i+1})}^{G_{\delta_i}}$-names
$\name{s}$ and $\rho$ in $M[G_{\delta_i}]$ for  element of $\IP_{[\delta_i,\delta_{i+1})}^{G_{\delta_i}}$
and $\IP_{[\delta_{i+1},\delta)}^{G_{\delta_i}}$, and names $\name{n}^j_{i+1}$ for natural numbers
such that
\begin{eqnarray}  \label{long}
 \nonumber \pi \forces_{\IP_{[\delta_i,\delta_{i+1}})^{G_{\delta_i}}} \name{s}\in M[G_{\delta_i}]\cap \Gamma^{G_{\delta_i}}_{[\delta_i,\delta_{i+1})}\wedge\name{s}\leq
\name{p}'\uhr [\delta_i,\delta_{i+1})^{G_{\delta_i}}\wedge\name{s}\forces_{\IP_{[\delta_i,\delta_{i+1}})^{G_{\delta_i}}}\\
 \rho\leq \name{p}'\uhr [\delta_{i+1},\delta)^{G_{\delta_i}}\wedge\rho\forces_{\IP_{[\delta_{i+1}},\delta)^{G_{\delta_i}}}\\ \nonumber
\forall j\in 2 \: |[f(\name{n}^j_{i+1}), f(\name{n}^j_{i+1}+1))\cap \name{Y}_{i+1}|\geq 2.
\end{eqnarray}
Using the maximality principle again, we can find
$\IP_{\delta_i}$-names for the objects appearing in
equation~(\ref{long}) such that $q_i$ forces this equation. We shall
use the same notation for these names. It remains to set
$q_{i+1}=q_i\bigvid \pi$ and
$\name{p}_{i+1}=\name{p}'_i\uhr\delta_i\bigvid\name{s}\bigvid\rho$
and note that they together with the names $\name{n}^j_{i+1}$, $j\in
2$, satisfy $(i)$-$(v)$ for $i+1$.
\end{proof}

By a \emph{Miller tree} we understand a subtree $T$ of $\w^{<\w}$
consisting of increasing  finite sequences such that the following
conditions are satisfied:
\begin{itemize}
\item
 Every $t \in T$ has an extension $s\in T$ which
is splitting in $T$ , i.e., there are more than one immediate
successors of $s$ in $T$;
\item If $s$ is splitting in $T$, then it
has infinitely many successors in $T$.
\end{itemize}
 The \emph{Miller forcing} is the
collection $\mathbb M$ of all Miller trees ordered by inclusion,
i.e., smaller trees carry more information about the generic. This
poset was introduced in \cite{Mil84}.
 For a Miller tree $T$ we
shall denote  the set of all splitting nodes of $T$ by $\spl(T)$.
$\spl(T)$ may be written in the form $\bigcup_{i\in\w}\spl_i(T)$,
where
$$ \spl_i(T)=\{t\in\spl(T)\: :\: |\{s\in\spl(T):s\subsetneq t\}|=i\}.$$
If $T_0, T_1\in\mathbb M$, then $T_1\leq_i T_0$ means $T_1\leq T_0$
and $\spl_i(T_1)=\spl_i(T_0)$. It is easy to check  that for any
sequence $\la T_i:i\in\w\ra\in\mathbb M^\w$, if $T_{i+1}\leq_i T_i$
for all $i$, then $\bigcap_{i\in\w}T_i\in\mathbb M$.

For a node $t$ in a Miller tree $T$ we denote by $T_t$ the set
$\{s\in T : s$ is compatible with $t\}$. It is clear that $T_t$ is
also a Miller tree.

\begin{lemma} \label{well-spl-miller}
The Miller forcing $\mathbb M$ is well-splitting.
\end{lemma}
\begin{proof}
Let $N$ be an elementary submodel of $H(\theta)$ and $T\in\mathbb
M\cap N$. Let $\{\name{Y}_i:i\in\w\} $ be an enumeration of all
$\mathbb M$-names for infinite subsets of $\w$ which are elements of
$N$, in which every such name appears infinitely often. Let also
$\{D_i:i\in\w\}$ be an enumeration of all open dense subsets of
$\mathbb M$ which belong to $N$. Suppose that $f\in\w^\w$
well-splits $N$. We shall inductively construct a sequence $\la
T_i:i\in\w\ra$ such that $T_{i+1}\leq_i T_i$ and
$T_\infty=\bigcap_{i\in\w}T_i$ is as required. Set $T_0=T$ and
suppose that $T_i$ has already been constructed. Moreover, we shall
assume that $(T_i)_t\in N$ for all $t\in\spl_i(T_i)$.
 Let $\{t_j:j\in\w\}$ be a bijective enumeration of
$\spl_i(T_i)$. For every $j$ and $k\in\w$ such that $t_j\bigvid k\in
T_i$ fix a decreasing sequence $\la S^{i,j,k}_{n}:n\in\w\ra\in N$ of
elements of $D_i$ below $(T_i)_{t_j\vid k}$ such that each
$S^{i,j,k}_n$ decides some $a^{i,j,k}_n\in [\w]^n$ to be  the set of
the first $n$ many elements of $\name{Y_i}$. Thus
$Y^{i,j,k}:=\bigcup_{n\in\w}a^{i,j,k}_n\in N\cap [\w]^\w$, and hence
there are $E_p\ni m^{i,j,k}_{n,p}\geq i$ such that
$$ |[ f(m^{i,j,k}_{n,p}), f(m^{i,j,k}_{n,p}+1 ))\cap Y^{i,j,k}_n|\geq 2 $$
for all $p\in 2$. Let $n(i,j,k)$ be such that
$$Y^{i,j,k}\cap\max_{p\in 2} f(m^{i,j,k}_{n(i,j,k),p}+1 ) \subset  a^{i,j,k}_{n(i,j,k)}$$
and set
$$T_{i+1}=\bigcup\{ S^{i,j,k}_{n(i,j,k)}:j\in\w, t_j\bigvid k\in T_i\}.$$
This completes our inductive construction of the fusion sequence
$\la T_i:i\in\w\ra$. We claim that $T_\infty$ is as required. First
of all,  $T_\infty $ is $(N,\mathbb M)$-generic  because the
collection $\bigcup\{ S^{i,j,k}_{n(i,j,k)}:j\in\w, t_j\bigvid k\in
T_i\}$ is a subset of $D_i$ and predense below $T_{i+1}$ (and hence
also below $T_\infty$). Now fix a $\mathbb M$-name $\name{Y}\in N$ for an
element of $[\w]^\w$ and suppose to the contrary, that there exist
$i\in\w$, $p\in 2$, and $R\leq T_\infty $  that forces
$|[f(m),f(m+1))\cap\name{Y}|\leq 1$ for all $E_p\ni m\geq i$.
Enlarging $i$, if necessary, we may assume that
$\name{Y}=\name{Y}_i$. Passing to a stronger condition, if
necessary, we may assume that $R\leq (T_i)_{t_j\vid k}$ for some
$i,j\in\w$ and $k$ such that $t_j\vid k\in T_i$. But then $R\leq
S^{i,j,k}_{n(i,j,k)}$, and the latter condition forces
$$ |[ f(m^{i,j,k}_{n,p}), f(m^{i,j,k}_{n,p}+1 ))\cap Y^{i,j,k}_n|=|[ f(m^{i,j,k}_{n,p}), f(m^{i,j,k}_{n,p}+1 ))\cap \name{Y}|\geq 2,$$
which leads to a contradiction since $m^{i,j,k}_{n,p}$ has been
chosen to be above $i$. This contradiction completes our proof.
\end{proof}

Every filter $\F$ gives rise  to a natural forcing notion $\bM_\F$
introducing a generic subset $X\in [\w]^\w$ such that $X\subset^* F$
for all $F\in\F$ as follows: $\bM_\F$ consists of pairs $\la s,F\ra$
such that $s\in [\w]^{<\w}$, $F\in\F$, and $\max s<\min F$. A
condition $\la s,F\ra$ is stronger than $\la t,G\ra$ if $F\subset
G$, $s$ is an end-extension of $t$, and $s\setminus t\subset G$.
$\mathcal M_\F$ is usually called \emph{Mathias forcing associated
with} $\F$. In the proof of the next lemma we shall work with clopen subsets of
$\mathcal P(\w)$ of the form $\uparrow s=\{X\subset\w:s\subset X\}$, where $s\in [\w]^{<\w}$.

\begin{lemma} \label{hur_filters_ws}
Suppose that $\F$ is a Hurewicz filter. Then $\bM_\F$ is well-splitting.
\end{lemma}
\begin{proof}
Suppose that $f$ well-splits $M\prec H(\theta)$, and $\F\in M$.
We shall prove that any $\la s_0, F_0\ra\in\bM_\F\cap M$
forces that $f$ well-splits $M[\Gamma]$. This suffices because all conditions in $\bM_\F$
are $(M,\bM_\F)$-generic. Suppose, contrary to our claim, that there exists
$\la s_1, F_1\ra\leq \la s_0, F_0\ra$
such that
\begin{eqnarray*}
\la s_1, F_1\ra\forces\exists \sigma \exists j \exists n_0\:\big(\sigma\in M\cap [\w]^\w\wedge j\in 2\wedge n_0\in\w\wedge\\
\wedge \forall n\in E_j\setminus  n_0\: (|[f(n), f(n+1))\cap\sigma|\leq 1)\big).
\end{eqnarray*}
Replacing $\la s_1, F_1\ra$ with a stronger condition, if necessary, we may fix $j\in 2$, $n_0\in\w$,
and a $\bM_\F$-name $\name{Y}\in M$ for an  infinite subsets of $\w$
such that
$$ \la s_1, F_1\ra\forces \forall n\in E_j\setminus  n_0\: (|[f(n), f(n+1))\cap\name{Y}|\leq 1). $$
Let  $\name{g}\in M$  be a name for a  function  such that $\name{g}(n)$ is forced to be the
$n$th element of $\name{Y}$.
For every $m\in\w$ let $\mathcal S_m$ be the set of those $ s\in
[F_0\setminus(\max s_1+1)]^{<\w}$ such that  there exist $F_s\in\F$
such that $\la  s_1\cup s , F_s \ra$ forces $ \name{g}(m+1)$ to be
equal to some $l_{s,m}\in\w $. It is clear that for every $F\in\F$ there
exists $s\in\mathcal S_m$ such that $s\subset F $. In other words, $
\U_m:=\{\uparrow s\colon   s \in \mathcal S_m\} $ is an open cover of
$\F$. Since $\F$ is Hurewicz,  there exists for every $m$ a finite
$\V_m\subset\U_m$ such that $\left\{\bigcup\V_m\colon m\in\w\right\}$ is a
$\gamma$-cover of $\F$. Let $\mathcal T_m\in [\mathcal S_m]^{<\w}$
be such that $\V_m=\{\uparrow s \colon  s \in \mathcal T_m\}$ and
$h(m)=\max\{ l_{s,m}\colon s\in\mathcal T_m\}+1$.
By elementarity, we can in addition assume that
$\la \U_m,\V_m, \mathcal S_m, \mathcal T_m:m\in\w\ra\in M$ as well as $h\in M$.

Set $h'(0)=h(0)$ and $h'(m+1)=h(h'(m))$ for all $m\in\w$.
Let $m_0$ be such that for every $m\geq m_0$ there exists $s\in\mathcal S_m\cap\mathcal P(F_1)$.
Set $n_1=\max\{n_0,h'(m_0)\}$.
Since $f$ well-splits $M$, the set
$I_j:=\{n\in E_j:|[f(n), f(n+1))\cap \mathrm{range}(h')|\geq 2\}$ is infinite. In particular,
it contains some $n_2>n_1$.  Thus
there exists $m\in\w$ such that
$$ f(n_2)\leq h'(m)<h'(m+1)=h(h'(m))<f(n_2). $$
 Now,  $ f$ is strictly increasing,  hence by the definition of $n_1$ we have that
$m>m_0$, and therefore there exist $s\in \mathcal S_{h'(m)}\cap\mathcal P(F_1)$.
Thus there exists
 $F_s\in\F$ such that
$$\la s_1\cup s, F_s\ra\forces\name{g}(h'(m)+1)=l_{s,h'(m)}<h(h'(m))<f(n_2).$$
Also,  $\la s_1\cup s, F_s\ra\forces\name{g}(h'(m))\geq h'(m)\geq f(n_2)$.
It follows from the above that
$\la s_1\cup s, F_s\ra$ forces that
$[f(n_2), f(n_2)+1)$ contains at least two elements of $\name{Y}$, namely
the $h'(m)$-th and $h'(m)+1$-st. On the other hand, $\la s_1\cup s, F_s\ra$
 is compatible with $\la s_1, F_1\ra$ because $s\subset F_1$ and $\max s_1<\min s$, $n_2>n_0$, $n_2\in E_j$,
and $\la s_1, F_1\ra$ forces $|[f(n), f(n+1))\cap\name{Y}|\leq 1$ for all $n\in  E_j\setminus n_0$.
In this way two compatible conditions $\la s_1, F_1\ra$ and $\la s_1\cup s, F\ra$ force
contradictory facts, which is impossible. This completes our proof.
\end{proof}

Let us mention that there is another property of posets $\mathcal M_\F$ for Hurewicz filters $\F$ which is preserved by \emph{finite} support iterations
and which guarantees that the ground model reals remain splitting and unbounded, see
\cite[Prop. 84]{BGHR??}.

\begin{corollary} \label{cohen}
The Cohen forcing is well-splitting.
\end{corollary}
\begin{proof}
The Cohen forcing is isomorphic to any countable atomless poset, in
particular to $\bM_{\hot Fr}$, where $\hot Fr$ is the Fr\'echet filter
consisting of all cofinite subsets of $\w$. It remains to note that
$\hot Fr$ is Hurewicz.
\end{proof}

Recall that a poset $\IP$ is \emph{$\w^\w$-bounding} if $\w^\w\cap V$ is
dominating in $V^{\IP}$.

\begin{lemma} \label{w_w}
Every proper $\w^\w$-bounding poset $\IP$ is well-splitting.
\end{lemma}
\begin{proof}
Let us fix a $\IP$-name $\name{Y}$ for an element of $[\w]^\w$,
$p\in\IP$,
and
pick a countable elementary submodel
$M$ of $H(\theta)$  such that $\IP,\name{Y},p\in M$.
 Suppose that  $f$
well-splits $M$ and $q\leq p$ is any  $(M, \IP)$-generic condition.
  Let  $\name{g}\in  M$ be a name for the function in
$\zrost$ which is the increasing enumeration of $\name{Y}$.
Since $\IP$
is  $\w^\w$-bounding and $q$ is $(M,\IP)$-generic, there exist
 $k_0\in\w$ and $h\in
M\cap\zrost$  such that $q\forces$ ``$\name{g}(k)< h(k)$ for all $k\geq k_0$''.
  Let $h_1\in M$ be  the following function: $h_1(0)=0$,
$h_1(n+1)=h(h(h_1(n)))+1$ for all $n\in\w$. Let $G$ be $\IP$-generic
containing $q$ and $Y,g$ be the evaluations of $\name{Y},\name{g}$
with respect to $G$, respectively.
 It
follows from the above that the set
$$I:=\{i\in E_0: |[f(i),f(i+1))\cap\mathrm{range}(h_1)|\geq 2\} $$
is  infinite.
For every $i\in I$ we can find $n_i\in \w $ such that
$h_1(n_i),h_1(n_i+1)\in [f(i), f(i+1))$. Thus if $i\geq k_0$ then we have
\begin{eqnarray*}
 f(i)\leq h_1(n_i)\leq g(h_1(n_i))< h(h_1(n_i))\leq g (h(h_1(n_i)))<\\
< h(h(h_1(n_i)))=h_1(n_i+1)<f(i),
\end{eqnarray*}
and hence
$ |[f(i),f(i+1))\cap Y |\geq 2 $ because $g(h_1(n_i)),  g (h(h_1(n_i)))$
belong to the latter intersection. Therefore
in $V[G]$ we have $I\subset \{i\in E_0: |[f(i),f(i+1))\cap Y |\geq 2\} $.
Since $G\ni q$ was chosen arbitrarily, we can conclude that
$q$ forces the set
$$ \{n\in E_0: |[f(n),f(n+1))\cap \name{Y}|\geq 2\} $$
to be infinite. Analogously we can get that
$q$ forces also the set
$$ \{n\in E_1: |[f(n),f(n+1))\cap \name{Y}|\geq 2\} $$
to be infinite, which completes our proof.
\end{proof}

Summarizing the results proved in  this section we get the following

\begin{theorem}\label{main}
The class of all well-splitting posets preserves ground model reals splitting and unbounded, is closed
under  countable support iterations, and includes $\w^\w$-bounding, Cohen, Miller, and
Mathias forcing associated to filters with the Hurewicz covering properties.
\end{theorem}


\begin{thebibliography}{ChGP??}

\bibitem{Abr10}  Abraham, U.,
\emph{Proper Forcing}, in: \textit{Handbook of Set Theory} (M.\
Foreman, A.\ Kanamori, and M.\ Magidor, Eds.), Springer, Dordrecht 2010, pp.
333--394.



\bibitem{Bla10}
Blass, A.,
\emph{Combinatorial cardinal characteristics of the continuum},
in: \textit{Handbook of Set Theory} (M.\ Foreman, A.\ Kanamori, and M.\ Magidor, Eds.),
Springer, Dordrecht, 2010, pp. 395--491.

\bibitem{Bre98} Brendle, J., {\it Mob families and mad families},
  Arch. Math. Logic  \textbf{37}  (1998),   183--197.


\bibitem{Bre17} Brendle, J., {\it Some problems in forcing theory: large continuum and
generalized cardinal invariants,} RIMS Kokyuroku 2042
(Infinite Combinatorics and Forcing Theory, T. Yorioka, ed.),
2017, 1--16.


\bibitem{BGHR??} Brendle, J.; Guzm\'an, O.;  Raghavan, D.;
 Hru\v{s}\'{a}k, M., {\it   Combinatorial properties of MAD families}, preprint.


\bibitem{BreRag14} Brendle, J.; Raghavan, D., {\it Bounding, splitting, and almost disjointness,}
Ann. Pure Appl. Logic \textbf{165} (2014),  631--651.

\bibitem{ChoGuzHru16}
Chodounsk\'{y}, D.; Guzm\'{a}n, O.; Hru\v{s}\'{a}k, M., {\it Mathias-Prikry and Laver type forcing; summable ideals, coideals, and $+$-selective filters,} Arch. Math. Logic \textbf{55} (2016), 493--504.


\bibitem{ChoRepZdo15} Chodounsk\'{y}, D.; Repov\v{s}, D.; Zdomskyy, L.,
{\it  Mathias forcing and combinatorial covering properties of filters},
  J. Symb. Log.  \textbf{80}  (2015),   1398--1410.



\bibitem{Hur27} Hurewicz, W.,
{\it \"Uber Folgen stetiger Funktionen,} Fund.  Math. \textbf{9}  (1927), 193--204.


\bibitem{COC2}  Just, W.; Miller, A.W.; Scheepers, M.; Szeptycki, P.J.,
{\it The combinatorics of open covers. II,}  Topology Appl.
\textbf{73}  (1996),   241--266.

\bibitem{Laf96}
Laflamme, C., {\it Filter games and combinatorial properties of strategies}, in: {\it Set theory} (Boise, ID, 1992--1994). Contemporary Mathematics  192,  American Mathematical Society, Providence, Rhode Island 1996, pp. 51--67.

\bibitem{Mil84}  Miller, A., {\it Rational perfect set forcing},
in: \textit{Axiomatic Set Theory} (J. Baumgartner, D. A. Martin, S.
Shelah, Eds.), Contemporary Mathematics 31, American Mathematical
Society, Providence, Rhode Island, 1984, pp. 143--159.

\bibitem{SchTal10}
Scheepers, M.; Tall, F., {\it Lindel\"of indestructibility, topological games and selection principles,} Fund. Math. \textbf{210} (2010), 1--46.

\bibitem{She84} Shelah, S., {\it   On cardinal invariants of the continuum.}in: \textit{Axiomatic Set Theory} (J. Baumgartner, D. A. Martin, S.
Shelah, Eds.), Contemporary Mathematics 31, American Mathematical
Society, Providence, Rhode Island, 1984, pp. 183--207.

\bibitem{She_cup} Shelah, S., {\it
Two cardinal invariants of the continuum ($\hot d<\hot a$) and FS linearly ordered iterated forcing},
Acta Math. \textbf{192} (2004),  187--223.

\bibitem{Vau90}  Vaughan, J.,
  {\it Small uncountable cardinals and topology},
  in: {\it  Open Problems in Topology} (J. van Mill, G.M. Reed,
Eds.),  Elsevier Sci. Publ., Amsterdam 1990, pp. 195--218.

\bibitem{Zdo??} Zdomskyy, L., {\it  Bredle's proof of the consistency of $\hot b< \hot a$, without ranks, games, and Cohen reals,} Visnyk of the Lviv Univ. Series Mech. Math. \textbf{89} (2020), 5--10.

\end{thebibliography}
\end{document}